\documentclass{amsart}

\usepackage{xcolor}
\usepackage{latexsym}
\usepackage{pdfsync}
\usepackage{amsthm}
\usepackage{amssymb}
\usepackage{amsmath}
\usepackage{amsfonts}
\usepackage{mathrsfs}
\usepackage[all]{xy}
\usepackage{texdraw}
\usepackage{graphicx}
\usepackage{psfrag}
\usepackage{verbatim}
\usepackage{subfig}    % For side-by-side figures with subcaptions
\usepackage{hyperref}

\newtheorem{theorem}{Theorem}[section]
\newtheorem{proposition}[theorem]{Proposition}

\newtheorem{corollary}[theorem]{Corollary}
\theoremstyle{definition}
\newtheorem{definition}[theorem]{Definition}

\newtheorem{example}[theorem]{Example}

\newcommand{\newword}[1]{\textit{#1}}
\newcommand{\N}{\mathbb{N}}

\newcommand{\sub}{\mathcal{D}}

\renewcommand{\L}{\mathsf{L}}
\newcommand{\R}{\mathsf{R}}
\newcommand{\rat}{\mathcal{R}}
\newcommand{\alp}{\mathcal{A}}
\newcommand{\leta}{a}
\newcommand{\letb}{a'}
\newcommand{\letter}{a}
\newcommand{\T}{\mathcal{T}}
\newcommand{\tape}{\tau}
\newcommand{\tapeother}{\tau'}
\newcommand{\states}{\mathcal{S}}
\newcommand{\state}{s}
\newcommand{\stateother}{s'}
\newcommand{\Z}{\mathbb{Z}}
\newcommand{\trans}{T}
\newcommand{\transfunc}{\tau}

\newcommand{\enc}{\epsilon}

\newcommand{\set}[2]{\{ #1 \mid #2 \}}
\newcommand{\bigset}[2]{\bigl\{ #1 \;\bigr|\; #2 \bigr\}}
\newcommand{\halt}{\mathcal{H}}
\newcommand{\haltinv}{\overline{\mathcal{H}}\rule{0pt}{0pt}}

\newcommand{\fT}{f_{\scriptscriptstyle\T}}

\begin{document}

\title[Some Undecidability Results]{Some Undecidability Results for Asynchronous Transducers and the Brin-Thompson Group~$2V$}

\author{James Belk}

\author{Collin Bleak}

\begin{abstract}Using a result of Kari and Ollinger, we prove that the torsion problem for elements of the Brin-Thompson group~$2V$ is undecidable.  As a result, we show that there does not exist an algorithm to determine whether an element of the rational group $\mathcal{R}$ of Grigorchuk, Nekrashevich, and Sushchanskii has finite order.  A modification of the construction gives other undecidability results about the dynamics of the action of elements of~$2V$ on Cantor Space.  Arzhantseva, Lafont, and Minasyanin prove in 2012 that there exists a finitely presented group with solvable word problem and unsolvable torsion problem. To our knowledge, $2V$ furnishes the first concrete example of such a group, and gives an example of a direct undecidability result in the extended family of R.~Thompson type groups.
\end{abstract}

\maketitle

\section{Introduction}

If $G$ is a finitely presented group, the \newword{torsion problem} for $G$ is the problem of deciding whether a given word in the generators represents an element of finite order in~$G$.  Like the word and conjugacy problems, the torsion problem is not solvable in general~\cite{BBN}.  Perhaps more surprising is the fact that there exist finitely presented groups with solvable word problem and unsolvable torsion problem.  This result was proven by Arzhantseva, Lafont, and Minasyanin in 2012~\cite{ALM}, but they did not give a specific example of such a group.

In the 1960's, Richard J.~Thompson introduced a family of three groups $F$, $T$, and $V$, which act by homeomorphisms on an interval, a circle, and a Cantor set, respectively.  These groups have a remarkable array of properties: for example, $T$ and~$V$ were among the first known examples of finitely presented infinite simple groups.  Though Thompson and McKenzie used $F$ to construct new examples of groups with unsolvable word problem~\cite{ThompsonMcKenzie}, the groups $F$, $T$, and $V$ themselves have solvable word problem, solvable conjugacy problem, and solvable torsion problem~(see~\cite{BeMa}).

In 2004, Matt Brin introduced a family $\{nV\}_{n=1}^\infty$ of Thompson like groups, where $1V=V$ \cite{Brin}.  Each group $nV$ acts by piecewise-affine homeomorphisms on the direct product of $n$ copies of the middle-thirds Cantor set.  {These groups are all simple~\cite{Brin3} and finitely presented~\cite{Brin2,HennigMatucci}, and indeed they have type~$F_\infty$~\cite{BritaEtAl,FluchEtAl}. It follows from Brin's work that they all have solvable word problems.}  Our first main result is the following.

\begin{theorem}\label{thm:MainTheorem1}For $n\geq 2$, the Brin-Thompson group $nV$ has unsolvable torsion problem.
\end{theorem}

It is easy to show that $mV$ embeds in $nV$ for $m\leq n$, so it suffices to prove this theorem in the case where~$n=2$.  Our strategy is to use elements of $2V$ to simulate the operation of certain Turing machines.  In 2008, Kari and Ollinger proved \cite{KaOl} that there does not exist an algorithm to determine whether a given complete, reversible Turing machine has uniformly periodic dynamics on its configuration space.  We show that every such machine is topologically conjugate to an (effectively constructible) element of~$2V$, and therefore there does not exist an algorithm to determine whether a given element of $2V$ has finite order.

Our second main result concerns the periodicity problem for asynchronous transducers.  Roughly speaking, a \newword{asynchronous transducer} is a finite-state automaton that converts an input string of arbitrary length to an output string.  The transducer reads one symbol at a time, changing its internal state and outputting a finite sequence of symbols at each step.  Asynchronous transducers are a natural generalization of \newword{synchronous transducers}, which are required to output exactly one symbol for every symbol read.

Every transducer defines a \newword{rational function}, which maps the space of infinite strings to itself.  A transducer is \newword{invertible} if this function is a homeomorphism.  The group of all rational functions defined by invertible transducers is the \newword{rational group}~$\mathcal{R}$ defined by Grigorchuk, Nekrashevych, and Sushchanskii~\cite{GNS2000}.  Subgroups of $\mathcal{R}$ are known as \newword{automata groups}.

The idea of groups of homeomorphisms defined by transducers has a long history. Al\v{e}sin \cite{Aleshin72} uses such a group to provide a counterexample to the unbounded Burnside conjecture. Later, Grigorchuk uses automata groups to provide a 2-group counterexample to the Burnside conjecture \cite{Grigorchuk80}, and to construct a group of intermediate growth, settling a well-known question of Milnor \cite{Grigorchuk83}. In the last decade, the work of Bartholdi, Grigorchuk, Nekrashevich, Sidki, S\v{u}n\'ic, and many others have advanced the theory of automata groups considerably, and have brought these groups to bear on problems in geometric group theory, complex dynamics, and fractal geometry.

A transducer is \newword{periodic} if some iterate of the corresponding rational function is equal to the identity.  Our second main theorem is the following.

\begin{theorem}\label{thm:MainTheorem2}There does not exist an algorithm to determine whether a given asynchronous transducer is periodic.
\end{theorem}

We prove this result by showing that every element of the group $2V$ is topologically conjugate to a rational function defined by a transducer.  Since the torsion problem in $2V$ is undecidable, Theorem~\ref{thm:MainTheorem2} follows.

One important problem in the theory of automata groups is the \textit{finiteness problem}: given a finite collection of invertible transducers, is it possible to determine whether the corresponding rational homeomorphisms generate a finite group?  This question was posed by Grigorchuk, Nekrashevych, and Sushchanskii in~\cite{GNS2000}, and has since received significant attention in the literature. {Gillibert \cite{Gillibert2013} proved that it is undecidable whether the semigroup of rational functions generated by a given collection of (not necessarily invertible) transducers is finite, which our result implies as well.  Akhvai, Klimann, Lombardy, Mairesse and Picantin \cite{AkhaviKlimannLombardyMairessePicantin2012}, Klimann~\cite{Klimann2012}, and Bondarenko, Bondarenko,  Sidke, and Zapata \cite{BondarenkoBondarenkoSidkiZapata2013} have also obtained partial decidability or undecidability results in various contexts.}
Our result settles the question for asynchronous transducers.

\begin{theorem}\label{thm:MainTheorem3}The finiteness problem for groups generated by asynchronous automata is unsolvable.
\end{theorem}

This follows immediately from Theorem~\ref{thm:MainTheorem2}, which states that no such algorithm exists for the cyclic group generated by a single asynchronous automaton.  Note that the finiteness problem is still open for groups generated by synchronous automata, which includes all Grigorchuk groups, branch groups, iterated monodromy groups, and self-similar groups.

In the last section, we show how to simulate arbitrary Turing machines using elements of $2V$, and we use the construction to prove some further undecidability results for the dynamics of elements.  For example, we prove that there exists an element of $f\in2V$ with an attracting fixed points such that the basin of the fixed point is a noncomputable set.

It is an open question whether the group $2V$ has a solvable conjugacy problem, and our result does not settle the issue.  However, it does seem clear that the conjugacy problem in $2V$ must be considerably harder than in Thompson's group~$V$.  In particular, there can be no conjugacy invariant for $2V$ that gives a complete description of the dynamics, for it is not even possible to detect whether the dynamics are periodic!  This contrasts sharply with Thompson's group~$V$, in which such an invariant is easy to compute (see~\cite{BeMa}).

Based on this result, it seems likely that the conjugacy problem in $2V$ is undecidable.  If this is indeed the case, the group $2V$ may be useful for public-key cryptography~\cite{AAG,KoLee}.

\section{Turing Machines}

In this section we define Turing machines, reversible Turing machines, and complete reversible Turing machines.  Our treatment here is very similar to the one in~\cite{KaOl}, which is in turn based on the treatments in \cite{Kur} and~\cite{Mor}.

For the following definition, we fix two symbols $\L$ (for \newword{left}) and $\R$ (for \newword{right}), representing the two types of movement instructions for a Turing machine.

\begin{definition}A \newword{Turing machine} is an ordered triple $(\states,\alp,\trans)$, where
\begin{itemize}
\item $\states$ is a finite set of \newword{states},\smallskip
\item $\alp$ is a finite alphabet of \newword{tape symbols}, and\smallskip
\item $\trans\subseteq (\states\times\{\L,\R\}\times \states) \cup (\states\times\alp\times \states\times\alp)$ is the \newword{transition table}.
\end{itemize}
\end{definition}

A \newword{tape} for a Turing machine $\T=(\states,\alp,\trans)$ is any function $\tape\colon\Z\to\alp$, i.e.~any element of~$\alp^\Z$.  A \newword{configuration} of a Turing machine is a pair $(\state,\tape)$, where $\state$ is a state and $\tape$ is a tape.  The set $\states\times\alp^\Z$ of all configurations is called the \newword{configuration space} for~$\T$.

Each element of the transition table $T$ is called an \newword{instruction}.  There are two types of instructions:
\begin{enumerate}
\item An instruction $(\state,\delta,\stateother)\in \states\times\{\L,\R\}\times\states$ is called a \newword{move instruction}, with \newword{initial state}~$\state$, \newword{direction}~$\delta$, and \newword{final state}~$\stateother$.\smallskip
\item An instruction $(\state,\leta,\stateother,\letb)\in \states\times\alp\times\states\times\alp$ is called a \newword{write instruction}, with \newword{initial state}~$\state$, \newword{read symbol}~$\leta$, \newword{final state}~$\stateother$, and \newword{write symbol}~$\letb$.
\end{enumerate}
Together, the instructions of $T$ define a \newword{transition relation}~$\to$ on the configuration space~$\states\times\alp^\Z$.

Specifically, let $W\colon\alp^\Z\times\alp \to \alp^\Z$ and $M \colon \alp^\Z\times\{\L,\R\}\to\alp^\Z$ be the functions defined by
\[
W(\tape,\leta)(n) \;=\; \begin{cases}\leta & \text{if }n = 0, \\ \tape(n) & \text{if }n \ne 0,\end{cases}
\quad\text{and}\quad
M(\tape,\delta)(n) \;=\; \begin{cases}\tape(n-1) & \text{if }\delta = \L, \\ \tape(n+1) & \text{if }\delta = \R,\end{cases}
\]
for all $n\in\Z$.  That is, $W$ is the function that writes a symbol on the tape at position~$0$, while $M$ is the function that moves the head left or right by one step.  Using these functions, we can define the transition relation $\to$ for configurations:
\begin{enumerate}
\item Each move instruction $(\state,\delta,\stateother)\in T$ specifies that
 \[
 (\state,\tape) \;\to\; \bigl(\stateother,M(\tape,\delta)\bigr)
 \]
 for every tape $\tape\in\alp^\Z$.\smallskip
\item Each write instruction $(\state,\leta,\stateother,\letb)\in T$ specifies that
\[
(\state,\tape) \;\to\; \bigl(\stateother,W(\tape,\letb)\bigr)
\]
for every tape $\tape\in\alp^\Z$ for which $\tape(0) = \leta$.
\end{enumerate}
This completes the definition of $\to$, as well as the all of the basic definitions for general Turing machines.

We are interested in certain kinds of Turing machines:

\begin{definition}Let $\T=(\states,\alp,\trans)$ be a Turing machine.
\begin{enumerate}
\item We say that $\T$ is \newword{deterministic} if for every configuration $(\state,\tape)$, there is at most one configuration $(\stateother,\tapeother)$ so that $(\state,\tape)\to(\stateother,\tapeother)$.\smallskip
\item We say that $\T$ is \newword{reversible} if $\T$ is deterministic and for every configuration $(\stateother,\tapeother)$, there is at most one configuration $(\state,\tape)$ such that $(\state,\tape)\to(\stateother,\tapeother)$.\smallskip
\item We say that $\T$ is \newword{complete} if for every configuration $(\state,\tape)$, there is at least one configuration $(\stateother,\tapeother)$ such that $(\state,\tape)\to(\stateother,\tapeother)$.  (That is, $\T$ is complete if it has no halting configurations.)
\end{enumerate}
\end{definition}

Though we have defined these conditions using the configuration space $\states\times\alp^\Z$, they can be checked directly from the transition table~$\trans$ (see~\cite{KaOl}).

\begin{proposition}If\/ $\T$ is a complete, reversible Turing machine, then the transition relation $\to$ is a bijective function on the configuration space~$\states\times\alp^\Z$.
\end{proposition}
\begin{proof}It is clear that $\to$ defines an injective function.  The surjectivity follows from a simple counting argument on the transitions between states (see~\cite{Kur}).
\end{proof}

If $\T$ is a complete, reversible Turing machine, the bijection $F\colon \states\times\alp^\Z \to \states\times\alp^\Z$ defined by the transition relation is called the \newword{transition function} for~$\T$. We say that $\T$ is \newword{uniformly periodic} if there exists an $n\in\mathbb{N}$ such that $F^n$ is the identity function.  In~\cite{KaOl}, Kari and Ollinger prove the following theorem:

\begin{theorem}[Kari-Ollinger]
\label{thm:KariOllinger}It is undecidable whether a given complete, reversible Turing machine is uniformly periodic.\hfill\qedsymbol
\end{theorem}

We shall use this theorem to prove the undecidability of the torsion problem for elements of~$2V$.

\section{The Group $2V$}

In this section we define the Brin-Thompson group $2V$ and establish conventions for describing its elements.  Because we wish to view the Cantor set as the infinite product space $\{0,1\}^\infty$, our notation and terminology is slightly different from that of~\cite{Brin}.

Let $C$ be the Cantor set, which we identify with the infinite product space $\{0,1\}^\infty$, and let $\{0,1\}^*$ denote the set of all finite sequences of $0$'s and~$1$'s.  Given a finite sequence $\alpha\in\{0,1\}^*$, the corresponding \newword{dyadic interval} in $C$ is the set
\[
I(\alpha) \;=\; \bigset{\alpha\omega}{\omega\in \{0,1\}^\infty}.
\]
where $\alpha\omega$ denotes the concatenation of the finite sequence $\alpha$ with the infinite sequence~$\omega$.  Note that the Cantor set $C$ is itself a dyadic interval, namely the interval $I(-)$ corresponding to the empty sequence.  A \newword{dyadic subdivision} of the Cantor set $C$ is any partition of $C$ into finitely many dyadic intervals.

Let $C^2$ denote the Cartesian product $C\times C$.  A \newword{dyadic rectangle} in $C^2$ is any set of the form $R(\alpha,\beta) = I(\alpha)\times I(\beta)$, where $I(\alpha)$ and $I(\beta)$ are dyadic intervals.  A \newword{dyadic subdivision} of $C^2$ is any partition of $C^2$ into finitely many dyadic rectangles.  As discussed in~\cite{Brin}, every dyadic subdivision of $C^2$ has an associated \newword{pattern}, which is a subdivision of the unit square into rectangles.  For example, Figure~\ref{fig:SquarePatterns}(a) shows a dyadic subdivision of $C^2$ and the corresponding pattern.
\begin{figure}[t]
\centering
\setlength{\fboxsep}{8pt}
\subfloat[]{\includegraphics{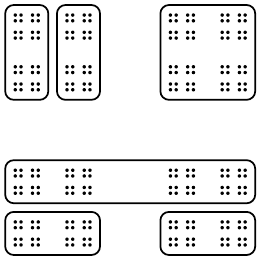}\quad\;\includegraphics{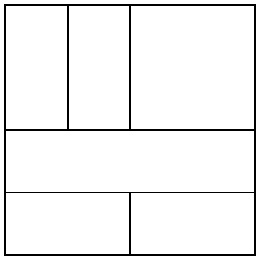}}
\hfill
\subfloat[]{\includegraphics{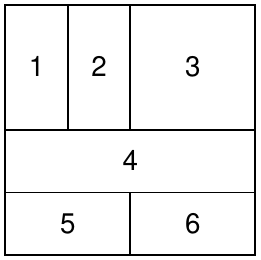}\quad\;\includegraphics{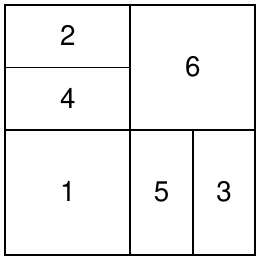}}
\caption{(a) A dyadic subdivision of $C^2$, and the corresponding pattern in the unit square. (b) A numbered pattern pair for an element of~$2V$.}
\label{fig:SquarePatterns}
\end{figure}

If $R(\alpha,\beta)$ and $R(\gamma,\delta)$ are dyadic rectangles, the \newword{prefix replacement function} $f\colon R(\alpha,\beta) \to R(\gamma,\delta)$ is the function defined by
\[
f(\alpha\psi,\beta\omega) \;=\; (\gamma\psi,\delta\omega),
\]
for all $\psi,\omega \in\{0,1\}^\infty$.  Note that this is a bijection between $R(\alpha,\beta)$ and $R(\gamma,\delta)$.

\begin{definition}The \newword{Brin-Thompson group $\boldsymbol{2V}$} is the group of all homeomorphisms $f\colon C^2\to C^2$ with the following property: there exists a dyadic subdivision $\sub$ of $C^2$ such that $f$ acts as a prefix replacement on each dyadic rectangle of~$\sub$.
\end{definition}

Note that the images $\{f(R) \mid R\in\sub\}$ of the rectangles of the dyadic subdivision $\sub$ are again a dyadic subdivision of~$C^2$.  Since there is only one prefix replacement mapping any dyadic rectangle to any other, an element of $2V$ is entirely determined by a pair of dyadic subdivisions, together with a one-to-one correspondence between the rectangles.  This lets us represent any element $f\in 2V$ by a pair of numbered patterns, as shown in Figure~\ref{fig:SquarePatterns}(b).

Note that the numbered pattern pair for an element $f\in 2V$ is not unique.  In particular, given any numbered pattern pair for~$f$, we can horizontally or vertically bisect a corresponding pair of rectangles in the domain and range to obtain another numbered pattern pair for~$f$.

Nonetheless, it is possible to compute effectively using numbered pattern pairs.  In~\cite{Brin}, Brin describes an effective procedure to compute a numbered pattern pair for a composition $fg$ of two elements of~$2V$, given a numbered pattern pair for each element.  Note that we can also effectively find a numbered pattern pair for the inverse of an element, simply by switching the numbered patterns for the domain and range.

\begin{proposition}Given numbered pattern pairs for two elements of~$2V$, there is an effective procedure to determine whether the two elements are equal.\label{prop:DetermineEqual}
\end{proposition}
\begin{proof}Let $f$ and $g$ be the two elements.  Using Brin's procedure, we can find a numbered pattern pair for $f^{-1} g$.  Then $f=g$ if and only if $f^{-1}g$ is the identity element, which occurs if and only if the numbered domain and range patterns for $f^{-1} g$ are identical.
\end{proof}

It is shown in \cite{Brin2} that the group $2V$ is finitely presented, with $8$ generators and $70$ relations.  The following proposition is implicit in~\cite{Brin} and~\cite{Brin2}.

\begin{proposition}The word problem is solvable in~$2V$.
\end{proposition}
\begin{proof}Given any two words, we can use Brin's procedure to construct numbered pattern pairs for the corresponding elements.  By Proposition~\ref{prop:DetermineEqual}, we can use these to determine whether the elements are equal.
\end{proof}

We will also need the following result.

\begin{proposition}Given a numbered pattern pair for an element $f\in 2V$, there is an effective procedure to find a word for~$f$.
\label{prop:EffectiveWord}
\end{proposition}
\begin{proof}Given a word $w$, we can determine whether $w$ represents $f$ by computing a numbered pattern pair for~$w$, and then comparing with $f$ using Proposition~\ref{prop:DetermineEqual}.  Therefore, we need only search through all possible words $w$ until we find one that agrees with~$f$.
\end{proof}

{See \cite{BurilloClearyMetric} for explicit bounds relating the word lengths of elements to the number of rectangles in a numbered pattern pair.}

\section{Turing Machines in $2V$}
\label{sec:TuringMachinesIn2V}

The goal of this section is to prove Theorem~\ref{thm:MainTheorem1}.  That is, we wish to encode any complete, reversible Turing machine as an element of~$2V$, in such a way that the Turing machine is uniformly periodic if and only if the element has finite order.

Let $\T = (\states,\alp,\trans)$ be a complete, reversible Turing machine.  Let $\{\state_1,\ldots,\state_m\}$ denote the states of~$\T$, and let $\{I(\sigma_1),\ldots,I(\sigma_m)\}$ be a corresponding dyadic subdivision of~$C$, where $\sigma_1,\ldots,\sigma_m\in\{0,1\}^*$.  Similarly let $\{\letter_1,\ldots,\letter_n\}$ denote the tape symbols for~$\T$, and let $\{I(\alpha_1),\ldots,I(\alpha_n)\}$ be a corresponding dyadic subdivision of~$C$, where $\alpha_1,\ldots,\alpha_n\in\{0,1\}^*$.

Given any infinite sequence $\omega\in\alp^\infty$ of symbols, we can encode it to obtain an infinite sequence $\enc(\omega) \in \{0,1\}^\infty$ of $0$'s and $1$'s as follows:
\[
\enc(\letter_{i_1},\letter_{i_2},\letter_{i_3},\ldots) \;=\; \alpha_{i_1}\alpha_{i_2}\alpha_{i_3} \cdots,
\]
That is, $\enc(\omega)$ is the infinite concatenation of the corresponding sequence of~$\alpha$'s.

\begin{proposition}
The function $\enc\colon \alp^\infty\to\{0,1\}^\infty$ defined above is a bijection.
\end{proposition}
\begin{proof}Let $\psi_0\in \{0,1\}^\infty$.  Since $\{I(\alpha_1),\ldots,I(\alpha_n)\}$ is a dyadic subdivision of~$C$, there exists a unique $i_1\in\{1,\ldots,n\}$ so that $\alpha_{i_1}$ is a prefix of~$\psi_0$.  Then $\psi_0 = \alpha_{i_1}\psi_1$ for some $\psi_1\in\{0,1\}^\infty$.  Then $\psi_1$ itself has some uniquely determined $\alpha_{i_2}$ as a prefix, and hence $\psi_1=\alpha_{i_2}\psi_2$ for some $\psi_2\in\{0,1\}^\infty$.  Continuing in this way, we can express $\psi_0$ uniquely as an infinite concatenation $\alpha_{i_1}\alpha_{i_2}\alpha_{i_3} \cdots$.  Then
\[
\enc^{-1}(\psi_0) \;=\; (\letter_{i_1},\letter_{i_2},\letter_{i_3},\ldots),
\]
which proves that $\enc$ is invertible.
\end{proof}

Now let $\Phi\colon \states\times\alp^\Z \to C^2$ be the \newword{configuration encoding} defined by
\[
\Phi(\state_i,\tape) \;=\; \bigl(\,\sigma_i\,\enc(\tape_L),\;\enc(\tape_R)\,\bigr)
\]
for every configuration $(\state_i,\tape)$, where
\[
\tape_L = \bigl(\tape(-1),\tape(-2),\tape(-3),\ldots\bigr)
\qquad\text{and}\qquad
\tape_R = \bigl(\tape(0),\tape(1),\tape(2),\ldots\bigr).
\]
That is, the first component of $\Phi(\state_i,\tape)$ encodes the state $\state_i$ as well as the left half of the tape~$\tape$, while the second component of $\Phi(\state_i,\tape)$ encodes the right half of the tape~$\tape$.  Clearly $\Phi$ is a bijection from the configuration space $\states\times\alp^\Z$ to~$C^2$.

\begin{theorem}Let $F\colon \states\times\alp^\Z\to\states\times\alp^\Z$ be the transition function for~$\T$.  Then $\fT = \Phi\circ F\circ\Phi^{-1}$ is an element of~$2V$.\label{thm:BigConjugacyTheorem}
\end{theorem}
\begin{proof}Since $\T$ is complete and reversible, we know that~$F$ is bijective, and therefore $\fT$ is bijective as well.  To show that $\fT \in 2V$, we need only demonstrate a dyadic subdivision $\sub$ of $C^2$ such that $\fT$ acts by a prefix replacement on each rectangle of the subdivision.

The subdivision $\sub$ consists of the following dyadic rectangles:
\begin{enumerate}
\item For each state $\state_i$ that is the initial state of a left move instruction, $\sub$~includes the rectangles $\bigl\{R(\sigma_i\alpha_k,-)\bigr\}_{k=1}^n$\smallskip
\item For each state $\state_i$ that is the initial state of a right move instruction, $\sub$~includes the rectangles $\bigl\{R(\sigma_i,\alpha_k)\bigr\}_{k=1}^n$.\smallskip
\item For each state $\state_i$ that is the initial state of write instructions, $\sub$~includes the rectangles $\bigl\{R(\sigma_i,\alpha_k)\bigr\}_{k=1}^n$.
\end{enumerate}
Note that, in each of the three cases, the given rectangles are a subdivision of~$R(\sigma_i,-)$.  It follows that $\sub$ is a dyadic subdivision of~$C^2$.  Moreover, it is easy to check that $\fT$ has the right form on each rectangle of~$\sub$.  In particular:
\begin{enumerate}
\item For a left move instruction $(\state_i,\L,\state_j)$, the formula for $\fT$ on each rectangle $R(\sigma_i\alpha_k,-)$ is
\[
\fT(\sigma_i\alpha_k\psi,\omega) \;=\; (\sigma_j\psi,\alpha_k\omega)
\]
for all $\psi,\omega\in \{0,1\}^\infty$.  That is, $\fT$ maps each $R(\sigma_i\alpha_k,-)$ to $R(\sigma_j,\alpha_k)$.\smallskip
\item For a right move instruction $(\state_i,\R,\state_j)$, the formula for $\fT$ on each rectangle $R(\sigma_i,\alpha_k)$ is
\[
\fT(\sigma_i\psi,\alpha_k\omega) \;=\; (\sigma_j\alpha_k\psi,\omega)
\]
for all $\psi,\omega\in \{0,1\}^\infty$.  That is, $\fT$ maps each $R(\sigma_i,\alpha_k)$ to $R(\sigma_j\alpha_k,-)$.\smallskip
\item For a write instruction $(\state_i,\letter_j,\state_k,\letter_\ell)$, the formula for $\fT$ on the rectangle $R(\sigma_i,\alpha_k)$ is
\[
\fT(\sigma_i\psi,\alpha_k\omega) \;=\; (\sigma_j\psi,\alpha_\ell \omega)
\]
for all $\psi,\omega\in \{0,1\}^\infty$.  That is, $\fT$ maps each $R(\sigma_i,\alpha_k)$ to $R(\sigma_j,\alpha_\ell)$.
\end{enumerate}
We conclude that $\fT\in 2V$.
\end{proof}

\begin{proposition}Given a complete, reversible Turing machine $\T = (\states,\alp,\trans)$, there exists an effective procedure for choosing an element~$\fT$ as defined above, and expressing it as a word in the generators of~$2V$.
\end{proposition}
\begin{proof}Note first that it is easy to choose the dyadic subdivisions $\{I(\sigma_1),\ldots,I(\sigma_m)\}$ and $\{I(\alpha_1),\ldots,I(\alpha_n)\}$.  For example, $(\sigma_1,\ldots,\sigma_m)$ could be the $m$'th term of the sequence
\[
(-),\quad (0,1),\quad (0,10,11),\quad (0,10,110,111),\quad (0,10,110,1110,1111),\quad \ldots,
\]
and $(\alpha_1,\ldots,\alpha_n)$ could be the $n$'th term of this sequence.

Once the subdivisions are chosen, we can use the formulas given in the proof of Theorem~\ref{thm:BigConjugacyTheorem} to construct a numbered pattern pair for the element~$\fT$.  Finally, we can use Proposition~\ref{prop:EffectiveWord} to compute a word for~$\fT$.
\end{proof}

\begin{proposition}The element $\fT$ has finite order if and only if $\T$ is uniformly periodic.
\end{proposition}
\begin{proof}Since $\fT = \Phi\circ F\circ\Phi^{-1}$, it follows that $(\fT)^p = \Phi\circ F^p\circ\Phi^{-1}$ for each~$p$, so $(\fT)^p$ is the identity if and only if $F^p$ is the identity function.
\end{proof}

This completes the proof of Theorem~\ref{thm:MainTheorem1}.  Combining this with the result of Kari and Ollinger (Theorem~\ref{thm:KariOllinger} above), we conclude that the torsion problem in $2V$ is undecidable.  The following theorem may shed some light on the nature of this result:

\begin{theorem}For each $n\in\mathbb{N}$, let\/ $\Omega(n)$ be the maximum possible order of a torsion element of $2V$ \!having at most $n$ rectangles in its numbered pattern pair.  Then\/ $\Omega$ is not bounded above by any computable function\/ $\mathbb{N}\to\mathbb{N}$.
\end{theorem}
\begin{proof}Suppose to the contrary that $\Omega$ were bounded above by a computable function $\gamma\colon\N\to\N$.  Then, given any element $f\in 2V$ with $n$ rectangles in its numbered pattern pair, it would be a simple matter to determine whether $f$ has finite order.  Specifically, we could first compute~$\gamma(n)$, and then compute the powers $f,f^2,f^3,\ldots,f^{\gamma(n)}$, and finally check to see if any of these is the identity.  Since the torsion problem in $2V$ is undecidable, it follows that no such function $\gamma$ exists.
\end{proof}

Thus the function $\Omega(n)$ must grow very quickly, e.g.~on the order of the busy beaver function (see~\cite{Rado}).

We end this section with an example that illustrates the construction of~$\fT$.

\begin{example}Consider the Turing machine $\T$ with four states $\{\state_1,\state_2,\state_3,\state_4\}$ and three symbols $\{\letter_1,\letter_2,\letter_3\}$, which obeys the following rules:
\begin{enumerate}
\item From state $\state_1$, move the head right and go to state $\state_2$.\smallskip
\item From state $\state_2$, read the input symbol $\tape(0)$:
\begin{enumerate}
\item If $\tape(0) = \letter_1$, then write $\letter_1$ and go back to state $\state_1$.
\item If $\tape(0) = \letter_2$, then write $\letter_3$ and go back to state $\state_1$.
\item If $\tape(0) = \letter_3$, then write $\letter_2$ and go to state $\state_3$.\smallskip
\end{enumerate}
\item From state $\state_3$, move the head left and go to state $\state_4$.\smallskip
\item From state $\state_4$, read the input symbol $\tape(0)$:
\begin{enumerate}
\item If $\tape(0) = \letter_1$, then write $\letter_1$ and go back to state $\state_3$.
\item If $\tape(0) = \letter_2$, then write $\letter_2$ and go to state $\state_1$.
\item If $\tape(0) = \letter_3$, then write $\letter_3$ and go back to state $\state_3$.
\end{enumerate}
\end{enumerate}
That is, the transition table for $\T$ is the set
\begin{multline*}
\qquad \{(\state_1,\R,\state_2),(\state_2,\letter_1,\state_1,\letter_1),(\state_2,\letter_2,\state_1,\letter_3),
(\state_2,\letter_3,\state_3,\letter_2), \\
(\state_3,\L,\state_4),(\state_4,\letter_1,\state_3,\letter_1),(\state_4,\letter_2,\state_1,\letter_2),
(\state_4,\letter_3,\state_3,\letter_3)\}\qquad
\end{multline*}
It is easy to check that $\T$ is complete and reversible.
\begin{figure}[t]
\begin{center}
\includegraphics{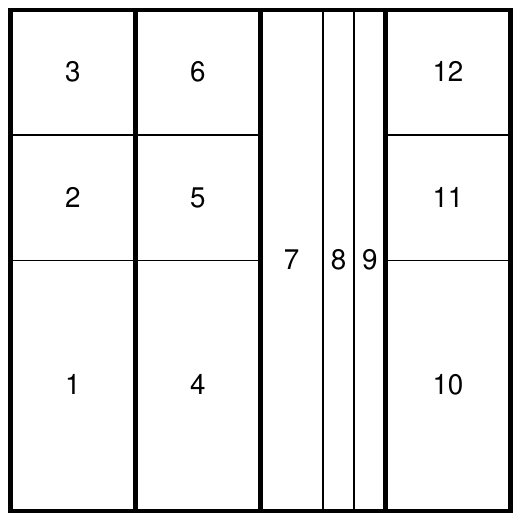}\qquad\includegraphics{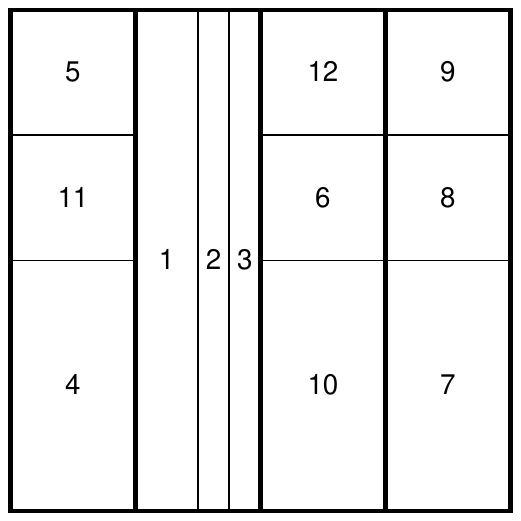}
\caption{The element of $2V$ corresponding to the Turing machine in Example~\ref{ex:TuringMachine}.  The four main vertical rectangles correspond to the four states $\{\state_1,\state_2,\state_3,\state_4\}$.}
\label{fig:TuringMachineElement}
\end{center}
\end{figure}\
To make a corresponding element of $2V$, let $(\sigma_1,\sigma_2,\sigma_3,\sigma_4)=(00,01,10,11)$, and let $(\alpha_1,\alpha_2,\alpha_3)=(0,10,11)$.  Then the resulting element $\fT\in 2V$ is shown in Figure~\ref{fig:TuringMachineElement}.
\label{ex:TuringMachine}
\end{example}

\section{Transducers}

{In this section we show that it is undecidable whether the rational function defined by a given asynchronous transducer has finite order (Theorem~\ref{thm:MainTheorem2}). This settles the finiteness problem for asynchronous automata groups (Theorem~\ref{thm:MainTheorem3}).}
  
{We begin by briefly reviewing the relevant facts about transducers.  See \cite{Gluskov63} for a thorough introduction to transducers, and \cite{GNS2000} for a discussion of transducers in the context of group theory.}

\begin{definition}
An asynchronous \newword{transducer} is an ordered quadruple $(\alp,\states,\state_0,\transfunc)$, where
\begin{itemize}
\item $\alp$ is a finite alphabet,\smallskip
\item $\states$ is a finite set of states,\smallskip
\item $\state_0 \in \states$ is the \newword{initial state}, and\smallskip
\item $\transfunc\colon \states \times \alp \to \states \times \alp^*$ is the \newword{transition function}, where $\alp^*$ denotes the set of all finite strings over~$\alp$.
\end{itemize}
\end{definition}

Given a transducer $(\alp,\states,\state_0,\transfunc)$ and an infinite \newword{input string} $\alpha_1\alpha_2\alpha_3\cdots\in\alp^\infty$, the corresponding \newword{state sequence} $\{s_i\}$ and \newword{output sequence} $\{\beta_i\}$ are defined recursively by
\[
(\state_i,\beta_i) \;=\; \transfunc(\state_{i-1},\alpha_i).
\]
The concatenation $\beta_1\beta_2\cdots$ of the output sequence is called the \newword{output string}.  This string is usually infinite, but will be finite if only finitely many $\beta_i$'s are nonempty.

We say that the transducer is \newword{nondegenerate} if every infinite input string results in an infinite output string.  In this case, the function $f\colon \alp^\infty \to \alp^\infty$ mapping each input string to the corresponding output string is called the \newword{rational function} defined by the given transducer.  {Rational functions are always continuous, and any invertible rational function is a homeomorphism.}

{For a given finite alphabet $\alp$, the \newword{rational group} $\rat$ is the group consisting of all invertible rational homeomorphisms of~$\alp^\infty$. It is proven in~\cite{GNS2000} that $\rat$ forms a group, and that the isomorphism type of $\rat$ does not depend on the size of~$\alp$, as long as $\alp$ has at least two letters.}

{Our goal in this section is to prove the following theorem.}

\begin{theorem}\label{thm:SubgroupOfRationalGroup}The rational group $\rat$ has a subgroup isomorphic to $2V$.
\end{theorem}

{Since $2V$ has unsolvable torsion problem, Theorems~\ref{thm:MainTheorem2} and \ref{thm:MainTheorem3} will follow immediately.}

{We begin with the following proposition, which will help us to} combine rational functions together.

\begin{proposition}\label{prop:TransducerSplicing}Let $\alp$ be a finite alphabet, let $\{\alpha_1,\ldots,\alpha_n\}$ be a complete prefix code over $\alp$, and let $f_1,\ldots,f_n\colon \alp^*\to\alp^*$ be rational functions.  Define a function $f\colon\alp^*\to\alp^*$ by
\[
f(\alpha_i\omega) = f_i(\omega)
\]
for all $i\in\{1,\ldots,n\}$ and $\omega\in\alp^*$.  Then $f$ is {a} rational {function.}
\end{proposition}
\begin{proof}Let $\{\beta_1,\ldots,\beta_m\}$ be the set of all proper prefixes of strings in $\{\alpha_1,\ldots,\alpha_n\}$, where $\beta_1$ is the empty string.  For each $i\in\{1,\ldots,n\}$, let $(\alp,\states_i,\state_{0i},\transfunc_i)$ be a transducer for~$f_i$.  Define a transducer $(\alp,\states,\state_0,\transfunc)$ as follows:
\begin{itemize}
\item The state set $\states$ is the disjoint union $\{\beta_1,\ldots,\beta_m\} \uplus \states_1 \uplus \cdots \uplus \states_n$.\smallskip
\item The initial state $\state_0$ is the empty string $\beta_1$.\smallskip
\item The transition function $\transfunc\colon \states\times \alp \to \states$ is defined by
\[
\transfunc(\state,\letter) \;=\;
\begin{cases}
(\beta_j,-) & \text{if }\state = \beta_i\text{ and }\beta_i\letter=\beta_j, \\
(\state_{0j},-) & \text{if }\state = \beta_i\text{ and }\beta_i\letter=\alpha_j, \\
\transfunc_i(\state,\letter) & \text{if }\state\in\states_i.
\end{cases}
\]
\end{itemize}
It is easy to check that the the rational function defined by this transducer is the desired function~$f$.
\end{proof}

Now consider the four-element alphabet $\alp=\{00,01,10,11\}$. Let $\pi\colon \alp^\infty \to C^2$ be the function defined by
\[
\pi(\epsilon_1\delta_1,\epsilon_2\delta_2,\epsilon_3\delta_3,\ldots) \;=\; (\epsilon_1\epsilon_2\epsilon_3\cdots,\delta_1\delta_2\delta_3\cdots).
\]
Given any function $f\colon C^2\to C^2$, let $f^\pi = \pi^{-1}\circ f\circ \pi$ denote the corresponding function $\alp^\infty \to \alp^\infty$.  {We shall prove that the mapping $f\mapsto f^{\pi}$ defines a monomorphism from $2V$ to~$\rat$.}

\begin{proposition}\label{prop:TransducersForParts}Let $\alpha,\beta\in \{0,1\}^*$, and let $\mu_{\alpha,\beta}\colon C^2\to C^2$ be the function
\[
\mu_{\alpha,\beta}(\psi,\omega) \;=\; (\alpha\psi,\beta\omega).
\]
Then $\mu_{\alpha,\beta}^\pi$ is a rational function.
\end{proposition}
\begin{proof}Suppose that $\alpha \in \{0,1\}$ has length $m$ and $\beta$ has length~$n$.  Consider the transducer $(\alp,\states,\state_0,\transfunc)$ defined as follows:
\begin{itemize}
\item The alphabet $\alp$ is $\{00,01,10,11\}$.\smallskip
\item The state set $\states$ is $\{0,1\}^m\times \{0,1\}^n$.\smallskip
\item The initial state $\state_0$ is $(\alpha,\beta)$.\smallskip
\item The transition function $\transfunc\colon \states\times \alp \to \states \times \alp^*$ is defined by
\[
\transfunc\bigl((\delta_1\cdots\delta_m,\epsilon_1\cdots\epsilon_n),\delta_{m+1}\epsilon_{n+1}\bigr) = \bigl((\delta_2\cdots\delta_{m+1},\epsilon_2\cdots\epsilon_{n+1}),\delta_1\epsilon_1\bigr).
\]
\end{itemize}
It is easy to check that the rational function defined by this transducer is~$\mu_{\alpha,\beta}^\pi$.
\end{proof}

\begin{proposition}If $f\in 2V$, then $f^\pi$ is a rational function.
\end{proposition}
\begin{proof}Let $\{R(\alpha_1,\beta_1),\ldots,R(\alpha_n,\beta_n)\}$ be a dyadic partition of $C^2$ so that $f$ is linear on each $R(\alpha_i,\beta_i)$.  By subdividing if necessary, we may assume that all of the strings $\alpha_1,\ldots,\alpha_n$ and $\beta_1,\ldots,\beta_n$ have the same length.  Let $R(\gamma_i,\delta_i) = f\bigl(R(\alpha_i,\beta_i)\bigr)$ for each~$i$, and let $\epsilon_i$ be the common initial prefix of $\pi\bigl(R(\alpha_i,\beta_i)\bigr)$.  Then $f^\pi$ is given by the formula
\[
f^\pi(\epsilon_i\omega) \;=\; \mu_{\gamma_i,\delta_i}^\pi(\omega)
\]
for each $i\in\{1,\ldots,n\}$ and each $\omega\in\{00,01,10,11\}^*$.  By Proposition~\ref{prop:TransducersForParts}, each of the functions $\mu_{\gamma_i,\delta_i}^\pi$ is  {rational}.  By Proposition~\ref{prop:TransducerSplicing}, it follows that~$f^\pi$ is {rational} as well.
\end{proof}

This completes the proof of Theorem~\ref{thm:SubgroupOfRationalGroup}, and hence Theorems~\ref{thm:MainTheorem2} and~\ref{thm:MainTheorem3}.

\section{Allowing Halting}

In this section, we briefly discuss how to simulate incomplete Turing machines using elements of~$2V$, and we sketch the proofs of some further undecidability results.  Similar results for general piecewise-affine functions can be found in~\cite{Blond}.

\begin{definition}Let $\T$ be a reversible Turing machine, and let $(\state,\tape)$ be a configuration for~$\T$.
\begin{enumerate}
\item We say that $(\state,\tape)$ is a \newword{halting configuration} if there does not exist any configuration $(\stateother,\tapeother)$ such that $(\state,\tape) \to (\stateother,\tapeother)$.\smallskip
\item We say that $(\state,\tape)$ is an \newword{inverse halting configuration} if there does not exist any configuration $(\stateother,\tapeother)$ such that $(\stateother,\tapeother) \to (\state,\tape)$.
\end{enumerate}
A Turing machine that reaches a halting configuration is said to \newword{halt}.
\end{definition}

If $\T$ is incomplete, then the transition function $F\colon \states\times\alp^\Z \to \states\times\alp^\Z$ on the configuration space is only partially defined, and is injective but not surjective.  Specifically, if $\halt$ is the set of halting configurations and $\haltinv$ is the set of inverse halting configurations, then $F$ restricts to a bijection $\halt^c \to \haltinv^c$, where $\halt^c$ and $\haltinv^c$ denote the complements of $\halt$ and $\haltinv$ in the configuration space.

Our construction of the corresponding element $\fT\in 2V$ is only a slight modification of the construction from Section~\ref{sec:TuringMachinesIn2V}.  To start, we subdivide $C^2$ into three dyadic rectangles
\[
R_\zeta = R(0,0),\qquad R_\T = R(-,1),\qquad R_\eta = R(1,0),
\]
We will use $R_\eta$ and $R_\zeta$ for halting and inverse halting, respectively, while $R_\T$ will be used for configurations of the Turing machine.

Now, let $\{\state_1,\ldots,\state_m\}$ denote the states of $\T$, and let $\{I(\sigma_1),\ldots,I(\sigma_m)\}$ be a corresponding subdivision of~$C$, where $\sigma_1,\ldots,\sigma_m\in\{0,1\}^*$.  Similarly, let $\{\letter_1,\ldots,\letter_n\}$ denote the tape symbols for~$\T$, and let $\{I(\alpha_1),\ldots,I(\alpha_n)\}$ be a corresponding a dyadic subdivision of~$C$, where $\alpha_1,\ldots,\alpha_n\in\{0,1\}^*$

Let $\enc\colon \alp^\infty\to \{0,1\}^\infty$ be encoding function derived from $(\alpha_1,\ldots,\alpha_n)$, and let $\Phi\colon \states\times\alp^\Z \to C^2$ be the configuration encoding defined by
\[
\Phi(\state_i,\tape) \;=\; \bigl(\,\sigma_i\,\enc(\tape_L),\;1\,\enc(\tape_R)\,\bigr)
\]
Note that $\Phi$ is no longer surjective---its image is the rectangle $R_\T$. Moreover, note that $\Phi(\halt)$ and $\Phi(\haltinv)$ are each the union of finitely many dyadic rectangles.

Let $\fT$ be any element of~$2V$ that satisfies the following conditions:
\begin{enumerate}
\item $\fT$ agrees with $\Phi\circ F\circ \Phi^{-1}$ on $\Phi(\halt^c)$.\smallskip
\item $\fT$ maps $R_\zeta$ bijectively onto  $R_\zeta \cup \Phi(\haltinv)$.\smallskip
\item $\fT$ maps $\Phi(\halt)\cup R_\eta$ bijectively onto $R_\eta$.
\end{enumerate}
Such an element can be constructed effectively.  For example, to construct the portion of $\fT$ on $R_\zeta$, we need only enumerate the dyadic rectangles $R_1,\ldots,R_k$ of~$\Phi(\haltinv)$, then choose a dyadic subdivision $\sub$ of $R_\zeta$ into $k+1$ rectangles, and finally choose a one-to-one correspondence between the rectangles of $\sub$ and $\{R_1,\ldots,R_k,R_\zeta\}$.

The element $\fT$ constructed above simulates the Turing machine $\T$, in the sense that the following diagram commutes:
\[
    \xymatrix@R=0.4in@C=0.2in@M=0.5em{
    \halt^c\ar[r]^F\ar[d]_\Phi & \haltinv^c\ar[d]^\Phi \\
    \Phi(\halt^c) \ar[r]_{\fT} & \Phi(\haltinv^c)
    }
\]
Whenever a configuration halts, $\fT$ maps the corresponding point to~$R_\eta$.  Similarly, $\fT$ maps points from $R_\zeta$ onto inverse halting configurations to ``fill in'' the bijection.

\begin{example}Consider the incomplete Turing machine $\T$ with two states $\{\state_1,\state_2\}$ and two symbols $\{\letter_1,\letter_2\}$, which obeys the following rules:
\begin{enumerate}
\item From state $\state_1$, move the head right and go to state $\state_2$.\smallskip
\item From state $\state_2$, read the input symbol $\tape(0)$:
\begin{enumerate}
\item If $\tape(0) = \letter_1$, then write $\letter_2$ and go back to state $\state_1$.
\item If $\tape(0) = \letter_2$, then halt.
\end{enumerate}
\end{enumerate}
That is, the transition table for $\T$ is the set
\[
\{(\state_1,\R,\state_2),(\state_2,\letter_1,\state_1,\letter_2)\}
\]
It is easy to check that $\T$ is reversible.  The halting configurations for $\T$ are
\[
\halt = \set{(\state_2,\tape)}{\tape\in\alp^\Z\text{ and }\tape(0) = \letter_2},
\]
and the inverse halting configurations are
\[
\haltinv = \set{(\state_1,\tape)}{\tape\in\alp^\Z\text{ and }\tape(0) = \letter_1}.
\]
To make a corresponding element of $2V$, let $(\sigma_1,\sigma_2)=(0,1)$ and $(\alpha_1,\alpha_2)=(0,1)$.  Then one possible choice for $\fT$ is shown in Figure~\ref{fig:TuringMachineElementHalt}.\label{ex:TuringMachineHalt}
\end{example}
\begin{figure}[t]
\begin{center}
\includegraphics{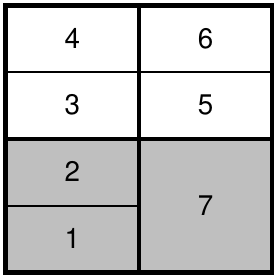}\qquad\includegraphics{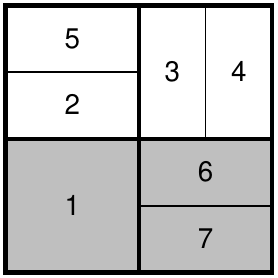}
\caption{The element of $2V$ corresponding to the Turing machine in Example~\ref{ex:TuringMachineHalt}.  The subrectangles $R_\zeta$ and $R_\eta$ are shown in gray, while $R_\T$ is shown in white.}
\label{fig:TuringMachineElementHalt}
\end{center}
\end{figure}

For our purposes, incomplete Turing machines are useful primarily because of their relation to the halting problem.  Before we can discuss this, we must restrict ourselves to a class of configurations that can be specified with a finite amount of information.  First, we fix a \newword{blank symbol} $\letter_1$ from~$\alp$, and we define a configuration $(\state_i,\tape)$ to be \newword{finite} if the tape $\tape$ has the blank symbol in all but finitely many locations.  Then any finite configuration $(\state_i,\tape)$ can be specified using only the state~$\state_i$ and some finite subsequence $\bigl(\tape(-n),\ldots,\tape(n)\bigr)$ of the tape whose complement consists entirely of blank symbols.

\begin{theorem}Any Turing machine can be effectively simulated by a reversible Turing machine.\label{thm:ReversibleUniversal}
\end{theorem}
\begin{proof}This was proven in \cite{Bennett1973} for a $3$-tape reversible Turing machine, and improved to a 1-tape, 2-symbol machine in~\cite{MSG}.
\end{proof}

\begin{corollary}It is not decidable, given a reversible Turing machine and a finite starting configuration, whether the machine will halt.\label{thm:NoHalt}
\end{corollary}
\begin{proof}This follows immediately from Theorem~\ref{thm:ReversibleUniversal} and Turing's Theorem on the unsolvability of the halting problem for general Turing machines.
\end{proof}

We wish to interpret this result in the context of $2V$.  We begin by defining points in $C^2$ that correspond to finite configurations.

\begin{definition}A point $(\psi,\omega)\in C^2$ is \newword{dyadic} if $\psi$ and $\omega$ have only finitely many~$1$'s.
\end{definition}

Note that a dyadic point in $C^2$ can be specified with a finite amount of information, namely the initial nonzero subsequences of $\psi$ and~$\omega$.  Assuming the sequence $\alpha_1$ corresponding to the blank symbol $\letter_1$ is a string of finitely many~$0$'s, the dyadic points in $R_\T$ are precisely the points that correspond to finite configurations of~$\T$.

\begin{theorem}It is not decidable, given an element $f\in 2V$\!, a dyadic point $p\in C^2$, and a dyadic rectangle $R\subseteq C^2$, whether the orbit of $p$ under $f$ contains a point in~$R$.
\end{theorem}
\begin{proof}Let $\T$ be a reversible Turing machine, let $(\state_i,\tape)$ be a starting configuration for $\T$, and let $\fT$ be the element constructed above.  Then $\T$ eventually halts starting at $(\state_i,\tape)$ if and only if the point $\Phi(\state_i,\tape)$ eventually maps into the rectangle $R_\eta$ under~$\fT$.  By Theorem~\ref{thm:NoHalt}, we cannot decide whether $\T$ will halt, so the given problem must be undecidable as well.
\end{proof}

\begin{theorem}It is not decidable, given an element $f\in 2V$ \!and two dyadic points $p,q\in C^2$, whether the orbit of $p$ under $f$ converges to~$q$.\label{thm:NoConvergenceCheck}
\end{theorem}
\begin{proof}Let $\T$ be a reversible Turing machine, and let $(\state_i,\tape)$ be a starting configuration for $\T$. Let $\fT$ be the element constructed above, with the function $\fT$ chosen so that $\fT(1\psi,0\omega) = \fT(10\psi,00\omega)$ for $(1\psi,0\omega)\in R_\eta$.  Note then that the orbit of every point in $R_\eta$ converges to $(1\overline{0},\overline{0})$.  Therefore, $\T$ eventually halts starting at $(\state_i,\tape)$ if and only if the orbit of $\Phi(\state_i,\tape)$ converges to $(1\overline{0},\overline{0})$.  By Theorem~\ref{thm:NoHalt}, we cannot decide whether $\T$ will halt, so the given problem must be undecidable as well.
\end{proof}

\begin{theorem}There exists an element $f\in 2V$ \!with an attracting dyadic fixed point $p\in C^2$ such that $B(p)\cap D$ is a non-computable set, where $B(p)$ is the basin of attraction of $p$ and $D$ is the set of dyadic points in~$C^2$.
\end{theorem}
\begin{proof}
It follows from Theorem~\ref{thm:ReversibleUniversal} that there exists a reversible Turing machine $\T$ that is computation universal, meaning that it can be used to simulate the operation any other Turing machine.  Such a machine has the property that, given a starting configuration $(\state_i,\tape)$, there does not exist an algorithm to determine whether $\T$ halts.  (That is, the halting problem is undecidable in the context of this one machine.)  If we use this machine to construct an element $\fT\in 2V$ as in the proof of Theorem~\ref{thm:NoConvergenceCheck}, then $\fT$ will have the desired property.
\end{proof}

\section*{Acknowledgments}
We would like to thank Matthew Brin, Rostislav Grigorchuk, Conchita Martinez-Perez, Francesco Matucci, Volodia Nekrashevich, and Brita Nucinkis for helpful conversations.

\bigskip

\def\cprime{$'$}

\providecommand{\bysame}{\leavevmode\hbox to3em{\hrulefill}\thinspace}

\providecommand{\MR}{\relax\ifhmode\unskip\space\fi MR}

\bibliographystyle{plain}

\end{document}